\documentclass[12pt]{article}
\usepackage{amsmath}
\usepackage{amsfonts}
\usepackage{amssymb}
\usepackage{mathrsfs}
\usepackage{epic}
\usepackage{curves}
\usepackage{amsthm}

\numberwithin{equation}{section}

\makeatletter
\newfont{\footsc}{cmcsc10 at 8truept}
\newfont{\footbf}{cmbx10 at 8truept}
\newfont{\footrm}{cmr10 at 10truept}
\makeatother 

\newtheorem{corollary}{Corollary}

\renewcommand{\leq}{\leqslant}
\renewcommand{\geq}{\geqslant}
\newtheorem{theorem}{Theorem}
\numberwithin{theorem}{section}
\numberwithin{corollary}{section}

\title{Combinatorial proofs of some Bell number formulas}
\author{%
\large Mark Shattuck\\
\small Mathematics Department\\[-0.8ex]
\small University of Tennessee\\[-0.8ex]
\small Knoxville, TN 37996-1320\\[-0.8ex]
\small\texttt{shattuck@math.utk.edu} \and
 }
\date{}

\begin{document}

\maketitle
\begin{abstract}
In this note, we provide bijective proofs of some identities involving the Bell number, as previously requested.  Our arguments may be extended to yield a generalization in terms of complete Bell polynomials.  We also provide a further interpretation for a related difference of Catalan numbers in terms of the inclusion-exclusion principle.
\end{abstract}

\section{Introduction}

By a \emph{partition} of a set, we will mean a collection of pairwise disjoint subsets, called \emph{blocks}, whose union is the set.  Let $\mathcal{P}_n$ denote the set of all partitions of $[n]=\{1,2,\ldots,n\}$.  Recall that the cardinality of $\mathcal{P}_n$ is given by the $n$-th Bell number $b_n$; see A000110 in \cite{Sl}.  In what follows, if $m$ and $n$ are positive integers, then let $[m,n]=\{m,m+1,\ldots,n\}$ if $m \leq n$, with $[m,n]=\varnothing$ if $m>n$.  Throughout, the binomial coefficient is given by $\binom{n}{k}=\frac{n!}{k!(n-k)!}$ if $0 \leq k \leq n$, with $\binom{n}{k}$ taken to be zero otherwise.

In \cite{QK}, combinatorial proofs were sought for some identities involving Bell numbers and binomial coefficients.  Relation \eqref{t1e1} below appears in a slightly different though equivalent form as Theorem 4.4 in \cite{QK}, where an algebraic proof was given, and identities \eqref{c1e1}, \eqref{c1e2}, and \eqref{c1e3} appear as Corollary 4.5.  It is the purpose of this note to provide the requested bijective proofs of \eqref{t1e1}--\eqref{c1e3}.  Our proof may be extended to obtain a more general relation involving Bell polynomials.
In the final section, we provide a further interpretation for the related alternating sum $\sum_{i=0}^n(-1)^{n-i}\binom{n}{i}c_i$, where $c_i$ denotes the $i$-th Catalan number.  In particular, we show that it can be thought of in terms of the inclusion-exclusion principle, as requested in \cite{QK}.

\section{Bijective proof of identities}

We first provide a  bijective proof of a formula relating an alternating sum to a positive one in the spirit of \cite{BQ}. (See also \cite[p. 59]{A}.)

\begin{theorem}\label{t1}
If $0 \leq j \leq n$, then
\begin{equation}\label{t1e1}
\sum_{i=0}^j (-1)^i\binom{j}{i}b_{n+1-i}=\sum_{k=0}^{n-j}\binom{n-j}{k}b_{n-k}.
\end{equation}
\end{theorem}
\begin{proof}
Given $0 \leq j \leq n$, let $\mathcal{A}=\mathcal{A}_{n,j}$ denote the set of ordered pairs $\lambda=(S,\pi)$, where $S \subseteq [j]$ and $\pi$ is a partition of $[n+1]-S$.  Define the sign of $\lambda$ to be $(-1)^{|S|}$. Note that the left-hand side of \eqref{t1e1} gives the total weight of all the members of $\mathcal{A}$.  We will define a sign-changing involution of $\mathcal{A}$ whose set of survivors has weight given by the right-hand side of \eqref{t1e1}.  To do so, let $\ell_0$ denote the largest $\ell \in [j]$, if it exists, such that either $\ell \in S$ or $\ell$ occurs as a singleton block within $\pi$.  Let $\lambda'$ denote the member of $\mathcal{A}$ obtained by either removing $\ell_0$ from $S$ and adding it to $\pi$ as the singleton block $\{\ell_0\}$ if $\ell_0 \in S$ or, vice-versa, if $\ell_0$ occurs within $\pi$ as a singleton.  Observe that the mapping $\lambda \mapsto \lambda'$ defines a sign-changing involution of $\mathcal{A}-\mathcal{A}^*$, where $\mathcal{A}^*$ denotes the subset of $\mathcal{A}$ consisting of all $\lambda$ of the form $\lambda=(\varnothing,\pi)$ in which $\pi$ is a member of $\mathcal{P}_{n+1}$ having no singleton blocks in $[j]$.  Note that each member of $\mathcal{A}^*$ has positive sign.

For example, suppose $n=8$, $j=4$, and $\lambda=(S,\pi) \in \mathcal{A}_{8,4}$, with $S=\{1,3\}$ and $\pi=\{2\},\{4,5\},\{6,8,9\},\{7\}$.  Then $\ell_0=3$ and $\lambda$ would be paired with $\lambda'=(S',\pi')$, where $S'=\{1\}$ and $\pi'=\{2\},\{3\},\{4,5\},\{6,8,9\},\{7\}$.

To complete the proof, we need to show that the number of partitions $\pi$ of $[n+1]$ having no singleton blocks among the elements of $[j]$ is given by $\sum_{k=0}^{n-j}\binom{n-j}{k}b_{n-k}$.  To count such $\pi$, consider the number $k$ of elements of $[j+1,n]$ belonging to the block containing $n+1$, where $0 \leq k \leq n-j$.  Having selected $k$ elements in $\binom{n-j}{k}$ ways to go in the block containing $n+1$, we then form a partition $\rho$ of the remaining $n-k$ elements of $[n]$ in $b_{n-k}$ ways.  Finally, we remove any singletons of $\rho$ belonging to $[j]$ and insert them as elements into the block containing $n+1$.  This yields a partition of $[n+1]$ containing no singletons in $[j]$.  Conversely, given any such member of $ \mathcal{P}_{n+1}$, one can remove all of the elements less than or equal $j$ in the block containing $n+1$ and add them back as singleton blocks.  Disregarding the block that now contains $n+1$ and possibly some elements of $[j+1,n]$ yields an arbitrary partition of size $n-k$, which completes the proof.
\end{proof}

We can also provide a bijective proof of the following corollary.

\begin{corollary}\label{c1}
If $j \geq0$, then
\begin{equation}\label{c1e1}
\sum_{i=0}^j (-1)^i \binom{j}{i}b_{j+1-i}=b_j
\end{equation}
and
\begin{equation}\label{c1e2}
\sum_{i=0}^j (-1)^i \binom{j}{i}b_{j+2-i}=b_j+b_{j+1}.
\end{equation}
If $j \geq 2$, then
\begin{equation}\label{c1e3}
\sum_{i=0}^j (-1)^i \binom{j}{i}b_{j-i}=\sum_{k=0}^{j-2}(-1)^k b_{j-1-k}.
\end{equation}
\end{corollary}
\begin{proof}
By the preceding argument, the left-hand side of \eqref{c1e1} counts the members of $\mathcal{P}_{j+1}$ in which no element of $[j]$ occurs as a singleton.  Equivalently, one may remove any singleton blocks from a member of $\mathcal{P}_j$ and form a new block with them together with the element $j+1$.  Similarly, the left-hand side of \eqref{c1e2} counts the members of $\mathcal{P}_{j+2}$ in which no element of $[j]$ occurs as a singleton.  On the other hand, one may remove any elements occurring as singletons within a member of $\mathcal{P}_j$ and form a new block with them together with the elements $j+1$ and $j+2$ or remove any singletons in $[j]$ from a member of $\mathcal{P}_{j+1}$ and insert them into a block with $j+2$.

To show \eqref{c1e3}, given $0 \leq k \leq j-2$, let $\mathcal{C}_{j-1-k}$ denote the subset of $\mathcal{P}_j$ whose members contain no singletons amongst $[j-k]$, with each element of $[j+1-k,j]$ occurring as a singleton, and let $\mathcal{D}_{j-1-k}$ denote the subset of $\mathcal{P}_j$ whose members contain no singletons amongst $[j-1-k]$, with each element of $[j-k,j]$ occurring as a singleton.  By the definitions, we have $\mathcal{D}_1=\varnothing$, $\mathcal{D}_{j-1-k}=\mathcal{C}_{j-2-k}$ if $0 \leq k \leq j-3$, and
$$b_{j-1-k}=|\mathcal{C}_{j-1-k}|+|\mathcal{D}_{j-1-k}|,$$
by \eqref{c1e1}, since $\mathcal{C}_{j-1-k}\cup\mathcal{D}_{j-1-k}$ is the same as the set of all members of $\mathcal{P}_{j-k}$ in which there are no singletons in $[j-1-k]$.  Thus, we have
\begin{align*}
\sum_{k=0}^{j-2}(-1)^kb_{j-1-k}&=\sum_{k=0}^{j-2}(-1)^k(|\mathcal{C}_{j-1-k}|+|\mathcal{D}_{j-1-k}|)\\
&=|\mathcal{C}_{j-1}|+(-1)^{j-2}|\mathcal{D}_1|=|\mathcal{C}_{j-1}|.
\end{align*}
On the other hand, we have $|\mathcal{C}_{j-1}|=\sum_{i=0}^{j}(-1)^i\binom{j}{i}b_{j-i}$, by the proof of \eqref{t1e1} above, since $\mathcal{C}_{j-1}$ is the set of all partitions of $[j]$ having no singleton blocks.  This completes the proof.

\end{proof}

\section{A generalization}

Let $B_n=B_n(t_1,t_2,\ldots,t_n)$ denote the $n$-th complete Bell polynomial; see, for example, \cite{Co}.  Recall that $B_n$ is obtained by assigning the weight $t_i$ to each block of a partition $\pi \in \mathcal{P}_n$ of size $i$, defining the weight of $\pi$ to be the product of the weights of its blocks, and then summing over all $\pi$.  It is well-known (see \cite{Co}) that $B_n$ is given by $\sum_{r=0}^nB_{n,r}(t_1,t_2,\ldots,t_n)$, where
$$B_{n,r}(t_1,t_2,\ldots,t_n)=\sum_{s}\frac{n!}{r_1!r_2!\cdots r_n!}\left(\frac{t_1}{1!}\right)^{r_1}\left(\frac{t_2}{2!}\right)^{r_2}\cdots\left(\frac{t_n}{n!}\right)^{r_n}$$
and the sum is taken over all nonnegative integer solutions $s=(r_1,r_2,\ldots,r_n)$ of $x_1+x_2+\cdots+x_n=r$ and $x_1+2x_2+\cdots +nx_n=n$. Note that $B_n=b_n$ when all weights are unity and $B_n=n!$ when $t_i=(i-1)!$ for all $i$.  Other sequences emerge from other specializations of the $t_i$; for example, $B_n=L_n$, the $n$-th Lah number, when $t_i=i!$ for all $i$ and $B_n=d_n$, the $n$-th derangement number, when $t_1=0$ and $t_i=(i-1)!$ for $i \geq 2$ (see A000262 and A000166, respectively,  in \cite{Sl}).

The following result generalizes \eqref{t1e1}, reducing to it when $t_i=1$ for all $i$.

\begin{theorem}\label{t2}
If $0 \leq j \leq n$, then
\begin{align}
\sum_{i=0}^j &(-1)^it_1^i\binom{j}{i}B_{n+1-i}\notag\\
&=\sum_{k=0}^{n-j}\sum_{\ell=0}^j \sum_{r=0}^{j-\ell} (-1)^rt_1^rt_{k+\ell+1}\binom{n-j}{k}\binom{j}{\ell}\binom{j-\ell}{r}B_{n-k-\ell-r}\label{t2e1}.
\end{align}
\end{theorem}
\begin{proof}
We extend the reasoning used to show \eqref{t1e1} above.  To the ordered pair $(S,\pi)$, we now assign the weight $$(-1)^{|S|}t_1^{|S|+s_1(\pi)}t_2^{s_2(\pi)}t_3^{s_3(\pi)}\cdots,$$ where $s_i(\pi)$ denotes the number of blocks of $\pi$ of size $i$.
Then the left-hand side of \eqref{t2e1} gives the total weight of all possible ordered pairs and the involution defined in the first paragraph of the proof of \eqref{t1e1} above is seen to reverse the weight.

To complete the proof, we argue that the total weight of all members of $\mathcal{P}_{n+1}$ having no singletons in $[j]$ is given by the right-hand side of \eqref{t2e1}.  In addition to choosing $k$ elements of $[j+1,n]$ to go in the same block as the element $n+1$ as before, we must now also select $\ell$ members of $[j]$ to go in this block.  Note that this creates a block of size $k+\ell+1$, whence the $t_{k+\ell+1}$ factor.  Once this is done, the remaining $n-k-\ell$ members of $[n]$ must be partitioned into sets such that none of the remaining $j-\ell$ members of $[j]$ occur as singletons.  Reasoning as before, this may be achieved in $\sum_{r=0}^{j-\ell}(-1)^rt_1^r\binom{j-\ell}{r}B_{n-k-\ell-r}$ ways, which completes the proof.
\end{proof}

Taking various specializations of the $t_i$ yields analogues of \eqref{t1e1} for other sequences such as $n!$, $L_n$, or $d_n$.

\section{An interpretation for a Catalan number difference}

We address here a related question which concerns a sum comparable to the left-hand side of \eqref{t1e1} above for the Catalan number $c_n=\frac{1}{n+1}\binom{2n}{n}$.  Let
$$K_n=\sum_{i=0}^n(-1)^{n-i}\binom{n}{i}c_i, \qquad n \geq 0.$$
In \cite{QK}, it was shown by induction and generating functions that $K_n$ enumerates a subset of a structure whose members the authors term \emph{non-interlocking, non-skipping $n$-columns} (denoted by NINS).  The problem of finding a direct argument for this using inclusion-exclusion was mentioned.

In fact, a clearer combinatorial interpretation for the numbers $K_n$ may be given.
Suppose $\Pi=B_1/B_2/\cdots \in \mathcal{P}_n$ has its blocks arranged in increasing order of smallest elements.  Recall that $\pi$ may be represented, canonically, as a \emph{restricted growth sequence} $\pi=\pi_1\pi_2\cdots\pi_n$, wherein $i \in B_{\pi_i}$ for each $i$.  See, e.g., \cite{Mi}.  For instance, the partition $\Pi=126/359/4/78$ would be represented as $\pi=112321442$.  A partition is said to be \emph{non-crossing} (see \cite{Kl}) if its canonical form contains no subsequences of the form $abab$ where $a<b$ (i.e., if it avoids all occurrences of the pattern $1212$).  By the definitions from \cite{QK}, an NINS $n$-column is seen to be equivalent to a non-crossing partition of the same length since, when an $n$-column is non-skipping, it can be shown that the interlocking property is equivalent to the non-crossing property.  Thus, Theorem 3.2 in \cite{QK} is equivalent to the following combinatorial interpretation of $K_n$:
\begin{align*}
&\qquad\text{The number of non-crossing partitions of length $n$ contain-}\\
&\qquad\text{ing no two equal adjacent letters (where the first and last}\\
&\qquad\text{letters are considered adjacent) is given by $K_n$.}
\end{align*}

In \cite{QK}, it was requested to provide a direct combinatorial proof of Theorem 3.2.  Using our interpretation for $K_n$ in terms of non-crossing partitions, one may provide such a proof as follows.   Note that within a  non-crossing partition represented sequentially, any letter whose successor is the same (as well as a possible 1 at the end) is extraneous concerning the containment or avoidance of the pattern $1212$.  Thus, one may cover any positions containing letters whose successor is the same (as well as a possible $1$ at the end) and consider the avoidance problem on the remaining uncovered positions.  Suppose one covers exactly $n-i$ positions.  Then the remaining letters constitute a partition of size $i$ avoiding the pattern $1212$, and it is well-known that such partitions are enumerated by $c_{i}$ (see \cite{Kl}).  When one uncovers the repetitive letters, no occurrences of $1212$ are introduced.  That $K_n$ counts all non-crossing partitions of length $n$ having no two adjacent letters the same now follows from an application of the inclusion-exclusion principle.  In fact, it is further seen that the sum $\sum_{i=0}^j(-1)^i\binom{j}{i}c_{n-i}$ counts all non-crossing partitions of length $n$ in which no letter equals its successor among the first $j$ positions in analogy to Theorem 2.3 in \cite{QK}.


\pagebreak
\bigskip
\hrule
\bigskip

\noindent 2010 {\em Mathematics Subject Classification:} Primary
05A19; Secondary 05A18.

\noindent {\em Keywords: Bell numbers, Bell polynomials, Catalan numbers, combinatorial proof.}

\bigskip
\hrule
\bigskip
\end{document}